\documentclass[12pt]{amsart}

\usepackage{amscd,amssymb,amsthm,verbatim}
\usepackage{enumerate}
\usepackage{bm}
\usepackage[margin=1in]{geometry}
\usepackage{mathrsfs, graphicx} 
\usepackage{stmaryrd}
\usepackage{epigraph}

\usepackage[%
bookmarks=true,
colorlinks,
linkcolor=blue,
urlcolor=blue,
citecolor=blue,
plainpages=false,
pdfpagelabels,
final,
breaklinks=true\between
]{hyperref}
\usepackage{hyperref}
\usepackage[numbers,sort&compress]{natbib}

\numberwithin{equation}{section}

\theoremstyle{plain}
\newtheorem{thm}{Theorem}[section]
\newtheorem{prop}[thm]{Proposition}
\newtheorem{lem}[thm]{Lemma}
\newtheorem{cor}[thm]{Corollary}

\theoremstyle{definition}
\newtheorem{defn}[thm]{Definition}
\theoremstyle{remark}
\newtheorem{rk}[thm]{Remark}
\newcommand{\ve}{\varepsilon}

\newcommand{\symD}{\mathbin{\Delta}}

\newcommand{\tr}{\operatorname{tr}}

\newcommand{\Ric}{\mathrm{Ric}}

\newcommand{\dist}{\mathrm{dist}}

\newcommand{\loc}{\mathrm{loc}}
\newcommand{\spt}{\operatorname{spt}}

\newcommand{\vol}{\operatorname{vol}}

\newcommand{\Rm}{\mathrm{Rm}}

\newcommand{\Div}{\operatorname{div}}

\newcommand{\csum}{\mathbin{\#}}

\DeclareFontFamily{U}{MnSymbolC}{}
\DeclareSymbolFont{MnSyC}{U}{MnSymbolC}{m}{n}
\DeclareFontShape{U}{MnSymbolC}{m}{n}{
	<-6>  MnSymbolC5
	<6-7>  MnSymbolC6
	<7-8>  MnSymbolC7
	<8-9>  MnSymbolC8
	<9-10> MnSymbolC9
	<10-12> MnSymbolC10
	<12->   MnSymbolC12}{}
\DeclareMathSymbol{\intprod}{\mathbin}{MnSyC}{'270}

\begin{document}
\title[Density, mass, and incompleteness]
{Density and positive mass theorems for incomplete manifolds}
\author{Dan A.\ Lee, Martin Lesourd, and Ryan Unger}

\address{Graduate Center and Queens College, City University of New York, 365 Fifth Avenue, New York, NY 10016}
\email{dan.lee@qc.cuny.edu}

\address{Black Hole Initiative, Harvard University, Cambridge, MA 02138}
\email{mlesourd@fas.harvard.edu}

\address{Department of Mathematics, Princeton University, Princeton, NJ 08544}
\email{runger@math.princeton.edu}

\maketitle

\begin{abstract}
For manifolds with a distinguished asymptotically flat end, we prove a density theorem which produces harmonic asymptotics on the distinguished end, while allowing for points of incompleteness (or negative scalar curvature) away from this end. We use this to improve the ``quantitative'' version of the positive mass theorem (in dimensions $3\leq n\leq 7$), obtained by the last two named authors with S.-T. Yau \cite{LUY21}, where stronger decay was assumed on the distinguished end. We also give an alternative proof of this theorem based on a relationship between MOTS and $\mu$-bubbles and our recent work on the spacetime positive mass theorem with boundary \cite{LLU}. 
\end{abstract}

\tableofcontents

\section{Introduction}

A natural question in scalar curvature geometry is whether minimal hypersurface and Dirac operator arguments can be effectively localized around a particular geometric feature. This is particularly important when working in ambient spaces that are noncompact, incomplete, or contain boundaries. In the case of minimal hypersurfaces, minimizing sequences are susceptible to various problems: they can escape every compact set and fail to converge, they can degenerate to something noncompact and unwieldy, or they can hit points of incompleteness and become singular. To get around these issues, M.~Gromov introduced the technique of $\mu$-bubbles as a way of forcing the minimizers to stay within a well understood and usually compact region of the ambient space  \cite{G96, G18}. With this technique, Gromov generalized R.~Schoen and S.-T. Yau's inductive descent method for closed manifolds to the setting of manifolds with boundaries. There is now a wealth of examples where $\mu$-bubbles have been used to study problems in scalar curvature curvature which previously seemed out of reach \cite{CL20, GromovFour, G20, Z20, Z21, LUY21, CLSZ21}.

 The second two authors of the present paper, together with Yau, investigated such a localization of the Riemannian positive mass theorem of Schoen and Yau \cite{SY79PMT}. By using a Plateau problem version of the $\mu$-bubble approach, they were able to prove a ``quantitative shielding'' version of the positive mass theorem for manifolds with one asymptotically Schwarzschild end and other arbitrary ends \cite[Theorem 1.7]{LUY21}, possibly incomplete and carrying negative scalar curvature. The positive mass theorem for complete manifolds with one asymptotically Schwarzschild end was obtained as a corollary of this quantitative version. 
 
The main result of the present paper is to relax the asymptotic decay required on the asymptotically flat end in the quantitative shielding theorem. Consequently, we also relax the decay assumption in the positive mass theorem obtained as a corollary, and moreover, we obtain a rigidity statement that does not follow from the results of \cite{LUY21}. We only assume Sobolev decay for the metric, that is, we assume  $g_{ij}-\delta_{ij}$ lies in a suitable weighted Sobolev space on the Euclidean end. These are the most general asymptotics under which the mass is invariantly defined \cite{BartnikMass, Chrusciel}. For precise definitions, see the beginning of Section \ref{preliminaries}. 

\begin{thm}[Quantitative shielding theorem]\label{thm1} 
Let $(M^n,g)$, $3\le n\le 7$, be an asymptotically flat manifold of Sobolev type $(p,q)$, with $p>n$ and $q>\frac{n-2}{2}$, 
not assumed to be complete or to have nonnegative scalar curvature everywhere. Let $U_0$, $U_1$, and $U_2$ be neighborhoods of an asymptotically flat end $\mathcal E$ such that  $\overline{U_2}\subset U_1$, $\overline{U_1}\subset U_0$, and $\overline {U_0 \setminus \mathcal E}$ is compact,
and let 
\[D_0=\dist_g(\partial U_0, U_1)\quad\text{and}\quad
D_1 =\dist_g(U_2,\partial U_1).\] 
If the following hold:
\begin{enumerate}
    \item $g$ has no points of incompleteness in $U_0$, 
    
    \item $R_g\ge 0$ on $U_0$, and 
    
    \item \label{scalar_bound} the scalar curvature satisfies the largeness assumption 
    \begin{equation}\label{largeness}
R_g>\frac{4}{D_0 D_1}\quad \text{on}\quad\overline{U_1}\setminus U_2,
    \end{equation}
\end{enumerate}
then the ADM mass of the asymptotically flat end $\mathcal E$ is strictly positive. 
\end{thm}
\begin{rk}
We also note that inequality \eqref{largeness} improves the one given in \cite[Theorem 1.7]{LUY21}, thanks to a slightly more precise argument.  
\end{rk} 

The proof of Theorem \ref{thm1} rests on the following \emph{density theorem}, which sharpens and generalizes previous results in this direction \cite{SY81ii, LeeParker, Kuwert}.

\begin{thm}[Density theorem]\label{DensityTheorem}
Let $(M^n,g)$, $3\le n$, be a Riemannian manifold, not assumed to be complete, with an  asymptotically flat end $\mathcal E$ of Sobolev type $(p,q)$, where $p>n$ and $q>\frac{n-2}{2}$. 
For any $\ve>0$, $\frac{n-2}{2}<q'<q$, and any compact set $K\subset M$, there exists another asymptotically flat metric $\tilde g$ of Sobolev type $(p,q')$ on $M$ with the following properties: 
\begin{enumerate}
    \item  $\tilde g$ is harmonically flat outside a bounded set in $\mathcal E$, that is, $\tilde g=u^\frac{4}{n-2}\overline g$, where $\overline g$ is the Euclidean metric on $\mathcal E$, and $u$ is a $\overline g$-harmonic function with expansion 
    \[u(x)=1+\frac{A}{|x|^{n-2}}+O_\infty(|x|^{-n-1}),\]

    \item The ADM mass of $\tilde g$ is $2A$ and we have 
   \[\|\tilde g-g\|_{W^{2,p}_{-q'}(K\cup \mathcal E)}+\|R_{\tilde g}-R_g\|_{L^1(K\cup\mathcal E)}+|2A-m_\mathrm{ADM}(\mathcal E, g)|<\ve,\] 
    
    \item $\sup_K |R_{\tilde g}-R_g|<\ve$,
    
    \item If $R_g(x)\ge 0$ at a point $x$, then $R_{\tilde g}(x)\ge 0$, and
    
    \item The metrics $g$ and $\tilde g$ are $\ve$-close as bilinear forms everywhere on $M$:
    \[(1-\ve)g\le \tilde g\le (1+\ve)g.\]
\end{enumerate}
\end{thm}

\begin{rk}
It will be clear from the proof of the density theorem that we can also accommodate the case when $M$ additionally has compact boundary components. 
\end{rk}

Theorem \ref{DensityTheorem} yields the positive mass theorem for complete asymptotically flat manifolds that have one distinguished asymptotically flat end (with Sobolev decay) and arbitrary other ends by application of \cite{LUY21}. The rigidity statement follows from essentially standard arguments within our analytic framework. 

\begin{thm}\label{pmt}
Let $(M^n,g)$, $3\le n\le 7$, be a complete manifold with nonnegative scalar curvature, and suppose that it has at least one asymptotically flat end $\mathcal E$ of Sobolev type $(p,q)$, $p>n$ and $q>\frac{n-2}{2}$. Then the ADM mass of $\mathcal E$ is nonnegative. Furthermore, if the mass is zero, then $(M,g)$ is isometric to Euclidean space. 
\end{thm}

As this paper was being completed, we learned of the concomitant and very interesting preprint of Jie Chen, Peng Liu, Yuguang Shi, and Jintian Zhu \cite{CLSZ21} which contains a positive mass theorem of this kind, along with the natural rigidity (see in particular \cite[Theorem 1.8]{CLSZ21}). In fact, their theorem applies to a more general class of asymptotically flat manifolds ``with fiber $F$''.\footnote{In \cite{CLSZ21}, the manifolds are assumed to be complete, and to have a distinguished end which is asymptotic to the product of a Euclidean end with a closed flat manifold, along with  incompressibility assumptions.
They also assume that the metric on the distinguished end satisfies pointwise $C^2$ decay rather than Sobolev.}
Their proof of positive mass involves two main steps: a density result \cite[Proposition 4.11]{CLSZ21} that applies to complete manifolds with a distinguished asymptotically flat end with fiber $F$ and nonnegative scalar curvature, followed by a reduction argument (different from Lohkamp's \@ \cite{L99}) to one of the geometric theorems about positive scalar curvature they prove in their paper. The statements and proofs of \cite[Proposition 4.11]{CLSZ21} and Theorem \ref{DensityTheorem} differ slightly, but they involve essentially the same ideas. As for the rigidity arguments of \cite[Theorem 1.8]{ CLSZ21} and Theorem \ref{pmt}, the proofs of Ricci flatness and the final contradiction are obtained slightly differently, though the more general topologies considered in \cite{CLSZ21} are an extra complication. We note the interesting possibility that the arguments of \cite[Theorem 1.8]{CLSZ21} can be combined with those of the current paper to yield generalizations of Theorem \ref{DensityTheorem} and Theorem \ref{thm1} that apply to the asymptotically flat ends ``with fiber $F$'' considered in \cite[Theorem 1.8]{CLSZ21}.

Theorem~\ref{thm1} also implies an inextendibility result: Given an asymptotically flat end $\mathcal E$ with nonnegative scalar curvature and negative mass, the positive mass theorem (Theorem~\ref{pmt}) tells us that it is impossible to extend $\mathcal E$ to be a complete manifold with nonnegative scalar curvature. The following corollary states that, in fact, there is a \emph{fixed distance} $D$ that puts a limit on how far we can extend the metric away from $\mathcal E$ before hitting either a point of incompleteness or a point of negative scalar curvature.

\begin{cor}\label{CZ1}
Let $(M^n,g)$,  $3\le n\le 7$, be a Riemannian manifold, not assumed to be complete, with an  asymptotically flat end $\mathcal E$ of Sobolev type $(p,q)$, where $p>n$ and $q>\frac{n-2}{2}$. 
 If $m_\mathrm{ADM}(\mathcal E,g)<0$, then there exists a constant $D$, depending only on $m_\mathrm{ADM}(\mathcal E,g)$ and $\|g-\delta\|_{W^{2,p}_{-q}(\mathcal E)}$, with the following property. In the $D$-neighborhood $N_D(\mathcal E)$ of $\mathcal E$, one or both of the following must be true:
\begin{enumerate}
 \item $R_g<0$ somewhere in $N_D(\mathcal E)$, or

    \item $N_D(\mathcal E)$ contains an incomplete point. 
\end{enumerate}
\end{cor}


Finally, in Section \ref{MOTS} we explain how $\mu$-bubbles can be viewed as a special case of marginally outer trapped surfaces (MOTS). In particular, we give a new proof of the quantitative shielding theorem, Theorem \ref{thm1}, that does not rely on the earlier resuls of~\cite{LUY21}. This new proof is achieved by constructing an unphysical second fundamental form $k$, inspired by the $\mu$-bubble construction, satisfying the dominant energy condition (DEC) and hiding the other ends, incompleteness, and negative scalar curvature behind an outer trapped surface. We can then apply the spacetime positive mass theorem \cite{SY81i, E13, EHLS} with a boundary, which was recently resolved (including rigidity) by the present authors \cite{LLU} (see also \cite{GallowayLee}). This relationship between $\mu$-bubbles and MOTS provides a link between problems in scalar curvature and initial data sets satisfying the DEC which may be useful in other contexts.  \\ \indent 
As an application of this method, we prove the following theorem where $M$ now has a \emph{non-mean convex}\footnote{Traditionally, the boundary mean curvature in the Riemannian PMT is measured with respect to the outward normal. However, in view of conventions used in the spacetime positive mass theorem with boundary (see \cite{GallowayLee, LLU}), we opt to state Corollary \ref{CZ2} with respect to the inner pointing normal.}  boundary with respect to the normal pointing out of the manifold. (The case of mean convex boundaries is already implicitly contained in the original work of Schoen--Yau \cite{SY79PMT, SY81ii}.)

\begin{thm}\label{CZ2}
Let $(M^n,g)$, $3\le n\le 7$, be a complete asymptotically flat manifold with nonempty compact boundary $\partial M$. Set $U_0=M$, and let $U_1$, $U_2$, $D_0$, $D_1$ be as in Theorem~\ref{thm1}, with the exception of item~\eqref{scalar_bound}. If we instead assume that
\begin{enumerate}
    \item \label{CZ2a} $R_g>\kappa$ on $\overline U_1\setminus U_2$ for a positive constant $\kappa$, 
    \item \label{CZ2b} the parameters satisfy $ \kappa < \frac{4}{D_0 D_1}$ (see Remark \ref{rk:a} below), and
    \item \label{CZ2c} the mean curvature of $\partial M$ with respect to the normal pointing into $M$ satisfies
\begin{equation}\label{Hineq}
H<\frac{2\kappa D_1}{4-\kappa D_0 D_1}, \end{equation}
\end{enumerate} 
then the ADM mass is strictly positive. 
\end{thm}

\begin{rk}\label{rk:a}
Fixing $D_0$ and $D_1$ in \eqref{Hineq} and letting $\kappa\to 0$ recovers the classical condition $H\le 0$. If  $ \kappa \ge \frac{4}{D_0 D_1}$, then Theorem \ref{thm1} applies without any consideration of the boundary mean curvature.
\end{rk}

\begin{rk}[Spinor methods]
Using Dirac operators on a space with a strong weight on the other ends, R.~Bartnik and P.~Chru\'sciel are able to prove a remarkable \emph{spacetime} positive mass theorem with arbitary ends under a spin assumption \cite{BC}.
 After \cite{LUY21} appeared, S.~Cecchini and R.~Zeidler revisted the Riemannian positive mass theorem in the spin setting \cite{CZ21}. They use Callias operators (i.e.\ Dirac operators with a potential) as a localization tool. They obtain analogues of all our results stated above, sans the density theorem \ref{DensityTheorem}, which is not needed for spin arguments. Also, their asymptotics are slightly stronger than ours, but we do not believe that this is essential. We would also like to point the reader to the very interesting paper \cite{CZ20}.
\end{rk}

\vspace{3mm}

\textbf{Acknowledgements}. M.L.\ thanks the Gordon and Betty Moore and
the John Templeton foundations for supporting the research carried out
at Harvard’s Black Hole Initiative. The authors thank Tin-Yau Tsang for several helpful comments that were incorporated into the second version of the paper.

\section{Asymptotic analysis in the presence of arbitrary ends} \label{preliminaries}

We begin this section with a precise statement of our definitions, which are slightly different than the usual ones because we allow for incompleteness and arbitrary ends.  

\begin{defn}\label{incomplete}
	Let $(X,d)$ be a metric space and $(\overline X,d)$ be its completion. For example, $\overline X$ can be constructed by taking appropriate equivalence classes of Cauchy sequences. A point in $\overline X\setminus X$ is called a called a \emph{point of incompleteness} for $X$. A set $S\subset X$ is said to be \emph{complete} if its closure in $X$ remains closed under the inclusion $X\to\overline X$.   
\end{defn}

\begin{defn}\label{strinfty}
Let $M^n$ be a noncompact manifold with a distinguished end $\mathcal E$. We say that $(M,g)$ possesses a \emph{structure of infinity} along $\mathcal E$ if $\mathcal E$ possesses no points of incompleteness and there exists a diffeomorphism 
\begin{equation}\Phi:\mathcal E\to \Bbb R^n\setminus B_{r_0}\end{equation} for some positive number $r_0$
and the coordinate norm $|x|$ diverges as we go out along the end. The set $\mathring M=M\setminus \mathcal E$ is called the \emph{core}. Note that in our definition, the core is not assumed to be compact and $(M,g)$ is not assumed to be complete. We will often identify $\mathcal E$ with the set $\{|x| \ge r_0\}$. The coordinates $x^i$ induce a natural flat metric on $\mathcal E$, which we extend arbitrarily to a complete metric on all of $M$ and denote by $\overline g$. 

We also allow for $M$ to have a boundary, but of course require $\partial M$ to not intersect~$\mathcal E$.
\end{defn}

\begin{defn}\label{sobolev}
Let $(M^n, g)$, $\mathcal E$, and $\Phi$ be as in the previous definition.  Let $N$ be a closed subset of $M$ which contains $\mathcal E$ and such that $N\setminus \mathcal E$ is compact. (For example, we might take $N=\mathcal E$.) Given $k\in\Bbb N$, $p\ge 1$, and $s\in\Bbb R$, we define the \emph{weighted Sobolev space} $W^{k,p}_s(N)$ to be the space of functions $u\in W^{k,p}_\loc(N)$ with finite norm 
\[\|u\|_{W^{k,p}_s(N)}=\|u\|_{W^{k,p}(N\setminus \mathcal E)}+\sum_{i=0}^k \|\partial^i u\|_{L^p_{s-i}(\mathcal E)},\]
where the weighted $L^p$ norm is defined by 
\[\|u\|_{L^p_s(\mathcal E)}=\left(\int_\mathcal E ||x|^{-s}u|^p \,\frac{dx}{|x|^n}\right)^\frac 1p.\]
Note that $r^s\notin L^p_s$ but $r^{s-\delta}\in L^p_s$ for any $\delta>0$. Note also that $L^p_{s'}\subset L^p_{s}$ if $s'\ge s$. We also remark that in our definition, the weighted spaces are to be constructed relative to the reference metric $\overline g$ and do not reference the (later) geometric metric $g$ at all. This is because $g$ will be changing at some points in the proof and we do not wish to have the norm changing as well. (This is only a minor point.)
\end{defn}

\begin{defn}\label{defAF}
Let $(M^n,g)$ be a noncompact smooth Riemannian manifold possessing a structure of infinity $\Phi$ along $\mathcal E$. Let $p>n$ and $q>\frac{n-2}{2}$. We say that $\mathcal E$ is \emph{asymptotically flat (AF) of Sobolev type $(p,q)$} if in the coordinates $x^i$ defined by $\Phi$, 
\[g_{ij}-\delta_{ij}\in W^{2,p}_{-q}(\mathcal E).\]
Furthermore, we assume that the scalar curvature of $g$, $R_g$, lies in $L^1(\mathcal E)$. A $(p,q)$ Sobolev asymptotically flat metric is also $(p,q')$ Sobolev asymptotically flat for any $q'<q$.

The \emph{ADM mass of $\mathcal E$} is defined by (see \cite{Lee} for history and discussion)
\begin{equation}m_\mathrm{ADM}(\mathcal E,g)=\lim_{r\to\infty}\frac{1}{2(n-1)\omega_{n-1}}\int_{|x|=r} (g_{ij,i}-g_{ii,j})\frac{x^j}{|x|}\,d\mu_{S_r,\overline g}.\end{equation}
The condition $R_g\in L^1(\mathcal E)$ guarantees that this is well-defined. 
\end{defn}

\begin{rk}
Our definition of asymptotically flat is the weakest ``standard" definition used in the positive mass literature. A stronger condition would be \emph{asymptotically flat with pointwise decay $q$}, that is,
\[g_{ij}(x)=\delta_{ij}+O_2(|x|^{-q}).\]
This will be of Sobolev type $(p,q)$ for any $p>n$ and $q'<q$. Yet another definition would be with weighted H\"older spaces. 

One reason to consider Sobolev decay is that Jang graphs in asymptotically flat initial data sets with pointwise decay are only Sobolev asymptotically flat \cite[Proposition~7]{E13}.
\end{rk}

Since $p>n$, the weighed Morrey inequality implies $g_{ij}-\delta_{ij}\in C^{1,\alpha}_{-q}$ for some $\alpha\in(0,1)$. We define the \emph{pointwise decay constant} $C_g<\infty$ of $g$ by \begin{equation}\label{pointwise}|g_{ij}-\delta_{ij}|+|x|^{-1}|\partial_kg_{ij}|\le C_g|x|^{-q}.\end{equation}
 
The following lemma is standard, cf.\ \cite[Lemma 3.35]{Lee}.

\begin{lem}\label{mass_conv}
Suppose that $g_i$ is a sequence of $(p,q)$ Sobolev asymptotically flat metrics converging in $W^{2,p}_{-q}(N)$ for $N$ as in Definition \ref{sobolev}. Assume that $R_{g_i}\to R_g$ in $L^1(\mathcal E)$. Then 
\[m_\mathrm{ADM}(\mathcal E, g_i)\to m_\mathrm{ADM}(\mathcal E,g).\]
\end{lem}

This motivates the following definition of closeness of asymptotically flat metrics. 

\begin{defn}\label{EpsClose}
Given $\ve>0$, 
we say that two $(p,q)$ Sobolev asymptotically flat metrics $g_1$ and $g_2$ on $M$ are \emph{$\ve$-close in the asymptotic topology}  if the following inequalities are satisfied:
\begin{enumerate}[(i)]
    \item $\|g_1-g_2\|_{W^{2,p}_{-q}(\mathcal E)}<\ve$,
    
    \item $\|R_{g_1}-R_{g_2}\|_{L^1(\mathcal E)}<\ve$,
    
    \item $|m_\mathrm{ADM}(\mathcal E, g_1)-m_\mathrm{ADM}(\mathcal E, g_2)|<\ve$,
    
    \item $(1-\ve) g_1\le g_2\le (1+\ve)g_1$ as bilinear forms, globally. 
\end{enumerate}
\end{defn}

To begin our construction, we define a distinguished function $\rho$ on $M$. 

\begin{lem}\label{defn_rho}
Let $(M,g,\mathcal E)$ be asymptotically flat, not assumed to be complete. There exists a $C^\infty$ proper function $\rho:M\to (0,\infty)$ which equals $|x|$ on $\mathcal E$ and for every sequence $\{p_i\}\subset M\setminus \mathcal E$ which eventually leaves every compact set, $\rho(p_i)\to 0$. We may additionally suppose that $\rho<r_0$ on $M\setminus\mathcal E$. 
\end{lem}
\begin{proof}
By a partition of unity argument, there exists a function $\sigma:M\to [1,\infty)$ such that $\sigma\to \infty$ outside of every compact set. We then interpolate between $\sigma^{-1}$ and $|x|$ near $|x|=r_0$ to obtain the desired function $\rho$.
\end{proof}

We use $\rho$ to construct a ``compact" exhaustion of the core which avoids incomplete points: For $\sigma>0$, let $M_\sigma=\{\rho \ge \sigma\}$. For $\sigma$ a regular value of $\rho$, $\partial M_\sigma$ is a smooth hypersurface. Then $M_\sigma$ is an asymptotically flat manifold with boundary and containing no incomplete points. 

The analytical content of this section is to construct solutions to certain Schr\"odinger equations on $M_\sigma$ and let $\sigma\to 0$. This is only possible because the potentials will vanish identically away from $M_{\sigma_0}$ for some fixed $\sigma_0>0$. We will need precise a priori estimates and so work very carefully and keep track of the constants. First, we note a standard Euclidean-type Sobolev inequality on $M_\sigma$.

\begin{lem}[Sobolev inequality for $M_\sigma$]\label{sobconst}
Let $(M^n, \overline{g},\mathcal E)$ be a fixed asymptotically flat manifold, and define $M_\sigma$ as above. Let $\Lambda>0$ be a fixed constant, and assume that $g$ is another metric on $M$ such that $\Lambda^{-1}\overline g\le g\le \Lambda \overline g$. Then for all $\sigma>0$, there exists a constant $C_\sigma$  depending only on $\sigma$ and $\Lambda$, such that if $u:M_\sigma\to \Bbb R$ is a $C^1$ function for which $du\in L^2(M_\sigma)$, then $u\in L^{\frac{2n}{n-2}}(M_\sigma)$ and 
\begin{equation}
\left(\int_{M_\sigma} |u|^\frac{2n}{n-2}\,d\mu_g\right)^{\frac{n-2}{n}}\le C_\sigma\int_{M_\sigma}|du|_g^2\,d\mu_g.\end{equation}
\end{lem}
\begin{proof}
For a fixed metric $\overline g$, this was proved in \cite{SY79PMT}, and then it is clear that one can switch from $\overline g$ to $g$ simply by multiplying $C_\sigma$ by a power of the uniform bound $\Lambda$,
\end{proof}

\begin{rk}
It is crucial for the method in this paper that this Sobolev inequality holds without a ``boundary condition" for $u$ on $\partial M_\sigma$. Indeed, imposing $u\to 0$ along $\mathcal E$ acts like a boundary condition. 
\end{rk}

We also require a scale-broken weighted elliptic estimate on $M_\sigma$.  

\begin{lem}\label{EllipticEstimate}
Let $(M^n,g,\mathcal E)$ be an asymptotically flat manifold of Sobolev type $(p,q)$ and $0<\delta<\sigma_0$. Then there exists a constant $C$, depending only on $n,p,q,\delta$, and the $C^{1,\alpha}_{-q}(M_{\sigma_0/2})$ norm of $g-\overline g$,
such that for any $u\in W^{2,p}_{-q}(M_{\sigma_0})$,
\begin{equation}\label{EE1}\|u\|_{W^{2,p}_{-q}(M_{\sigma_0})}\le C\left(\|\Delta_g u\|_{L^p_{-q-2}(M_{\sigma_0-\delta})}+\|u\|_{L^\frac{2n}{n-2}(M_{\sigma_0-\delta})}\right).\end{equation}
\end{lem}

We emphasize that this is an ``interior-type" estimate since $\overline M_{\sigma_0}\subset M_{\sigma_0-\delta}$. We also take this opportunity to define a cutoff function that will be used later as well. Let $\chi(|x|)$ be a radial cutoff in $\Bbb R^n$ which is one on the ball $B_1$ and zero outside the ball $B_2$. For $\lambda\ge r_0$ we define $\chi_\lambda$ on $M$ by setting it equal to $\chi(|x|/\lambda)$ on $\mathcal E$ and extending to the rest of the manifold by one. 

\begin{proof}[Proof of Lemma \ref{EllipticEstimate}]
Let $\lambda\ge r_0$ and set $u_0=\chi_\lambda u$, $u_1=(1-\chi_\lambda)u$, so that $u=u_0+u_1$. Then $u_1\in W^{2,p}_{-q}(\Bbb R^n)$, so we have the sharp estimate \cite[Theorem 0]{McO79}
\[\|u_1\|_{W^{2,p}_{-q}(\Bbb R^n)}\le C\|\Delta_{\overline g} u_1\|_{L^p_{-q-2}(\Bbb R^n)}.\]
Since $\Delta_g-\Delta_{\overline g}$ is the zero operator for $|x|\le \lambda$, we then have
\begin{align*}
    \|\Delta_{\overline g} u_1\|_{L^p_{-q-2}(\Bbb R^n)}&\le \|(\Delta_g-\Delta_{\overline g}) u_1\|_{L^p_{-q-2}(\Bbb R^n)}+\|\Delta_{g} u_1\|_{L^p_{-q-2}(\Bbb R^n)} \\
    &\le C\lambda^{-q} \|u_1\|_{W^{2,p}_{-q}(\Bbb R^n)}+\|\Delta_{g} u_1\|_{L^p_{-q-2}(\Bbb R^n)}.
\end{align*}
Choosing $\lambda$ large enough (depending only on the quantities listed in the statement of the lemma), we have 
\[\|u_1\|_{W^{2,p}_{-q}(\Bbb R^n)} \le C\|\Delta_{g} u_1\|_{L^p_{-q-2}(\Bbb R^n)}.\] 
Inserting the definition of $u_1$ into the right-hand side and carrying out the differentiations, we obtain error terms over the fixed compact set $K=\spt \nabla\chi_\lambda$. Lower order terms are moved to the left-hand side by interpolation, leaving us with 
\[\|u_1\|_{W^{2,p}_{-q}(\Bbb R^n)}\le C \left(\|\Delta_g u\|_{L^p_{-q-2}(M)}+\|u\|_{L^p(K)}\right).\] If $p\le \frac{2n}{n-2}$ (only possible when $n=3$ since $p>n$), we use H\"older's inequality to estimate $\|u\|_{L^p(K)}\le C\|u\|_{L^\frac{2n}{n-2}(K)}$.
If $p> \frac{2n}{n-2}$, we trivially estimate $\|u\|_{L^p(K)}\le C\|u\|_{L^\infty}$ and use the De Giorgi--Nash--Moser theorem to estimate 
\[\|u\|_{L^\infty(K)}\le C\left(\|\Delta_g u\|_{L^p(K')}+\|u\|_{L^\frac{2n}{n-2}(K')}\right),\]
where $K'$ is a slightly larger compact set containing $K$. The unweighted $L^p$ norm on the right-hand side can be absorbed into a weighted $L^p$ norm. Altogether, we have proved
\begin{equation}\label{211a}
    \|u_1\|_{W^{2,p}_{-q}(\Bbb R^n)}\le C \left(\|\Delta_g u\|_{L^p_{-q-2}(M)}+\|u\|_{L^\frac{2n}{n-2}(K)}\right). 
\end{equation}

The estimate
\begin{equation}\label{211b}
    \|u_0\|_{W^{2,p}_{-q}(M_{\sigma_0})}\le C \left(\|\Delta_g u\|_{L^p_{-q-2}(M)}+\|u\|_{L^\frac{2n}{n-2}(M_{\sigma_0-\delta})}\right). 
\end{equation}
for $u_0$ is much simpler since it vanishes on $\mathcal E$. It may be achieved by applying the $L^p$ theory for elliptic equations on the compact set $M_{\sigma_0-\delta}\setminus\{|x|\ge \lambda\}$. Combining \eqref{211a} and \eqref{211b} gives \eqref{EE1}, as desired. 
\end{proof}

The following is the main result of this section and can be compared to \cite[Lemma 3.2]{SY79PMT}.

\begin{prop}\label{Existence} Let $(M^n,g,\mathcal E)$ be a $(p,q)$ asymptotically flat manifold with $p>n$, $q> \frac{n-2}{2}$, and $\sigma_0>0$.  Let $V$ be a smooth function on $M$, and let $V^-$ denote its negative part. Assume that $\spt(V)\subset M_{\sigma_0}$ and $\|V^-\|_{L^\frac n2(M_{\sigma_0})}<C_{\sigma_0}^{-1}$, where $C_{\sigma_0}$ is a constant as in Lemma~\ref{sobconst} that works for this specific $g$).
 Then the Neumann problem
\[(-\Delta_g+V,\nu_g|_{\partial M_\sigma}): W^{2,p}_{-q}(M_\sigma)\to L^p_{-q-2}(M_\sigma)\times W^{2-\frac 1p,p}(\partial M_\sigma)\]
is an isomorphism for any $\sigma\in (0,\sigma_0)$ a regular value of $\rho$. 

There exist constants $\ve_0>0$ and $C$, depending only on $n,p,q,$
and the $C^{1,\alpha}_{-q}(M_{\sigma_0/2})$ norm of $g-\overline g$, such that if $\sigma\in(0,\frac{\sigma_0}{2})$, $f\in L^p_{-q-2}(M_{\sigma_0})\cap L^\frac{2n}{n+2}(M_{\sigma_0})$ is also supported in $M_{\sigma_0}$, 
$\|V^-\|_{L^\frac n2(M_{\sigma_0})}+\|V\|_{L^p_{-q-2}(M_{\sigma_0})}<\ve_0$, 
and $u\in W^{2,p}_{-q}(M_\sigma)$ solves 
\begin{align}
 \label{2a}   -\Delta_g u + Vu &=f \quad\text{in }M_\sigma\\
 \label{2b}  \nu_g(u) &=0 \quad\text{on }\partial M_\sigma,
\end{align}
then 
\begin{equation}\label{2d}\sup_{M_{\sigma}}|u|+\|u\|_{W^{2,p}_{-q}(M_{\sigma_0})}\le C\left(\|f\|_{L^p_{-q-2}(M_{\sigma_0})}+\|f\|_{L^\frac{2n}{n+2}(M_{\sigma_0})}\right).\end{equation}
\end{prop}
In practice, we will only apply this lemma in the situation where $f=-V$. 

\begin{rk}
To get a sense for how we will use this proposition to prove the density theorem, see the beginning of Section~\ref{sec:density}.
Our basic observation is that the equations to be solved have ``small" potentials which are identically zero away from $\mathcal E$. (See~\eqref{mainPDE1} below.) The control over $\sup_{M_{\sigma}}|u|$, rather than simply $\sup_{M_{\sigma_0}}|u|$ in Proposition~\ref{Existence} comes from the maximum principle.
This addresses the essential issue of completeness because we can make this as small as we like. Eichmair used a similar observation to prove a version of the positive mass theorem for manifolds with cylindrical ends \cite{E13}.
\end{rk}

\begin{proof}[Proof of Proposition \ref{Existence}]
Since the problem is self-adjoint, we only need to show that the operator has no kernel. Suppose $w\in W^{2,p}_{-q}(M_\sigma)$ satisfies 
\begin{align*}
    -\Delta_g w + Vw&=0 \quad\text{in }M_\sigma\\
   \nu_g(w) &=0 \quad\text{on }\partial M_\sigma.
\end{align*}
Invoking Lemma~\ref{sobconst} and then integrating by parts (which we justify in Lemma \ref{intbyparts} below), we have
\begin{align*}
  \left(\int_{M_{\sigma_0}} |w|^\frac{2n}{n-2} \,d\mu_g\right)^\frac{n-2}{n}  &\le C_{\sigma_0}\int_{M_{\sigma}}|dw|^2_g\,d\mu_g\\
   &=C_{\sigma_0}\int_{M_\sigma} -(\Delta_g w)w\,d\mu_g\\
    &=C_{\sigma_0}\int_{M_\sigma} -Vw^2\,d\mu_g\\
    &\le C_{\sigma_0}\int_{M_{\sigma_0}} V^- w^2\,d\mu_g\\
    &\le C_{\sigma_0}\left(\int_{M_{\sigma_0}}|V^-|^\frac{n}{2}\,d\mu_g\right)^\frac{2}{n}\left(\int_{M_{\sigma_0}}w^\frac{2n}{n-2}\,d\mu_g\right)^\frac{n-2}{n}.
\end{align*}
Our hypothesis that $\|V^-\|_{L^\frac n2(M_{\sigma_0})}<C_{\sigma_0}^{-1}$ implies that $w$ vanishes on~$M_{\sigma_0}$. But $w$ is harmonic on a neighborhood of $M_{\sigma}\setminus M_{\sigma_0}$ in $M_\sigma$, so it must vanish on  all of~$M_\sigma$. 

Now suppose $u$ satisfies \eqref{2a} and \eqref{2b}. 
To prove \eqref{2d}, we first note that since $u$ is harmonic on $M_\sigma\setminus M_{\sigma_0}$ and satisfies a Neumann condition on $\partial M_\sigma$, the Hopf lemma implies 
\[\sup_{M_\sigma}|u|=\sup_{M_{\sigma_0}}|u|,\]
and by the weighted Morrey inequality,
\[\sup_{M_{\sigma_0}}|u|\le  C\|u\|_{W^{1,p}_{-q}(M_{\sigma_0})}\le C\|u\|_{W^{2,p}_{-q}(M_{\sigma_0})},\]
where $C$ depends only on $\sigma_0$. We now apply the interior estimate \eqref{EE1} with $\delta=\sigma_0/2$ and note that $f$ and $V$ are supported in $M_{\sigma_0}$ to obtain 
\begin{align*}\sup_{M_\sigma}|u| + \|u\|_{W^{2,p}_{-q}(M_{\sigma_0})}
&\le C\left(\|Vu\|_{L^p_{-q-2}(M_{\sigma_0/2})}+\|f\|_{L^p_{-q-2}(M_{\sigma_0/2})}+\|u\|_{L^\frac{2n}{n-2}(M_{\sigma_0/2})}\right)\\
&\le C\left(\|V\|_{L^p_{-q-2}(M_{\sigma_0})}\sup_{M_{\sigma}}|u|+\|f\|_{L^p_{-q-2}(M_{\sigma_0})}+\|u\|_{L^\frac{2n}{n-2}(M_{\sigma_0/2})}\right).\end{align*}
By choosing $\ve_0$ small enough, $\|V\|_{L^p_{-q-2}}(M_{\sigma_0})$ will be small enough so that we have
\begin{equation}
    \sup_{M_\sigma}|u| + \|u\|_{W^{2,p}_{-q}(M_{\sigma_0})}
    \le 2C\left(\|f\|_{L^p_{-q-2}(M_{\sigma_0})}+\|u\|_{L^\frac{2n}{n-2}(M_{\sigma_0/2})}\right).
\end{equation}

It only remains to estimate $\|u\|_{L^\frac{2n}{n-2}(M_{\sigma_0/2})}$. Using the Sobolev inequality and integrating by parts as we did for $w$ above, 
\begin{align*}
  C_{\sigma_0/2}^{-1}\left(\int_{M_{\sigma_0/2}} |u|^\frac{2n}{n-2} \,d\mu_g\right)^\frac{n-2}{n}  &\le 
    \int_{M_{\sigma_0/2}} (f-Vu)u\,d\mu_g\\
    &\le \left(\int_{M_{\sigma_0/2}} |f|^\frac{2n}{n+2}\,d\mu_g\right)^\frac{n+2}{2n}\left(\int_{M_{\sigma_0/2}}|u|^\frac{2n}{n-2}\,d\mu_g\right)^\frac{n-2}{2n}\\
    &\quad\quad +\left(\int_{M_{\sigma_0/2}} |V^-|^\frac{n}{2}\,d\mu_g\right)^\frac{2}{n}\left(\int_{M_{\sigma_0/2}}|u|^\frac{2n}{n-2}\,d\mu_g\right)^\frac{n-2}{n}.
\end{align*}
\[C_{\sigma_0/2}^{-1}\|u\|_{L^\frac{2n}{n-2}(M_{\sigma_0/2})}^2\le \|f\|_{L^\frac{2n}{n+2}(M_{\sigma_0})}\|u\|_{L^\frac{2n}{n-2}(M_{\sigma_0/2})}
+\|V^-\|_{L^\frac{n}{2}(M_{\sigma_0})}\|u\|_{L^\frac{2n}{n-2}(M_{\sigma_0/2})}^2.\]

So long as  $\varepsilon_0<\tfrac{1}{2}C_{\sigma_0/2}^{-1}$, we can absorb the $V^-$ term to obtain
\begin{equation}\|u\|_{L^\frac{2n}{n-2}(M_{\sigma_0/2})}\le 2C_{\sigma_0/2}\|f\|_{L^\frac{2n}{n+2}(M_{\sigma_0})} \end{equation}
and the result follows.
\end{proof}

\begin{lem}\label{intbyparts}
Let $w\in W^{2,p}_{-q}(M_\sigma)$ for $p>n$. Then $dw\in L^2(M_\sigma)$ and if $\nu_g(w)=0$ on $\partial M_\sigma$, then 
\[\int_{M_\sigma}(-\Delta_g w)w\,d\mu_g=\int_{M_\sigma}|dw|^2_g\,d\mu_g.\]
\end{lem}
\begin{proof}
By Morrey's inequality, $w\in C^1_{-q}$ so both sides of the equality are defined. Furthermore, integrating by parts on the compact domain $\{\sigma\le \rho\le r\}$, we pick up a boundary term 
\[\int_{|x|=r}\nu_g(w)w\,d\mu_{S_r, g}.\]
By inspection the integrand is $O(r^{-2q-1})$, so it must disappear in the limit since $q>\frac{n-2}{2}$. 
\end{proof}

Because we will use them many times, we record some basic facts about conformal metrics constructed using the previous proposition. 

\begin{prop}\label{ConfFactConst}
Let $(M^n,g,\mathcal E)$ be a $(p,q)$ asymptotically flat manifold with $p>n$, $q> \frac{n-2}{2}$, and $\sigma_0>0$.  Let $V$ be a smooth integrable function on $M$ that is compactly supported in $M_{\sigma_0}$.  There exists a constant $\ve_0>0$, depending only on $n,p,q,$
and the $C^{1,\alpha}_{-q}(M_{\sigma_0/2})$ norm of $g-\overline g$, such that if
\[ \|V^-\|_{L^\frac n2(M_{\sigma_0})}+\|V\|_{L^p_{-q-2}(M_{\sigma_0})}+\|V\|_{L^\frac{2n}{n+2}(M_{\sigma_0})} <\ve_0,\]
then there exists a globally defined function $u$ on $M$ such that 
\[
 -a\Delta_g u + Vu  =0,     
\]
everywhere, where $a=4\frac{n-1}{n-2}$, such that $u-1\in W^{2,p}_{-q}(M_{\sigma_0})$. Moreover, $u$ has positive upper and lower bounds,
 and we can define the metric $\tilde g= u^\frac{4}{n-2}g$. This metric  $\tilde g$ is asymptotically flat of Sobolev type $(p,q)$, with scalar curvature
\begin{equation}
    R_{\tilde g}=(R_g-V)u^{-\frac{4}{n-2}}\label{conformalformula}
\end{equation}
and ADM mass
\begin{equation}
    m_\mathrm{ADM}(\mathcal E, \tilde g)=    m_\mathrm{ADM}( \mathcal E,g)-\frac{1}{2(n-1)\omega_{n-1}}\int_M Vu\,d\mu_g.\label{massformula}
\end{equation}
\end{prop}
\begin{proof}
We first invoke  Proposition \ref{Existence} with $f=-V$ to see that for any $\sigma\in(0,\sigma_0)$ that is a regular value of $\rho$, there exists solution $u_\sigma$ to the problem 
\begin{align}
 \label{2a-new}   -a\Delta_g u_\sigma + Vu_\sigma  &=0 \quad\text{in }M_\sigma\\
 \label{2b-new}  \nu_g(u_\sigma) &=0 \quad\text{on }\partial M_\sigma,\\
u_\sigma-1 &\in W^{2,p}_{-q}(M_\sigma)
\end{align}
Using the global estimates \eqref{2d} together with local elliptic theory, it follows that for some sequence of $\sigma$'s converging zero, the $u_\sigma$'s converge locally in $W^{2,p}$ to some globally define function $u$. By \eqref{2d} and smallness of $\ve_0$, we can ensure that $u$ has a positive upper and lower bound. (In fact, we can choose $\tfrac{1}{2}<u<\tfrac{3}{2}$.)
 The formula \eqref{conformalformula} follows from the standard formula for scalar curvature of a conformal metric. 

It is a standard fact that $(1+v)^\frac{4}{n-2}-1\in W^{2,p}_{-q}$ if $v\in W^{2,p}_{-q}$ \cite[Lemma 2.2(i)]{Kuwert}. We claim that $(u^\frac{4}{n-2}-1)g_{ij}\in W^{2,p}_{-q}$, for then
\[\tilde g_{ij}-\delta_{ij}=(u^\frac{4}{n-2}g_{ij}-1)g_{ij}+(g_{ij}-\delta_{ij})\in W^{2,p}_{-q}.\]
It is an easy matter to check that the claim is true, and thus $\tilde{g}$ is  asymptotically flat of Sobolev type $(p,q)$.

Since $V$ is integrable, \eqref{conformalformula} implies that $R_{\tilde g}$ is as well. To compute the mass of $\tilde g$, we compute the masses of the metrics $\tilde g_{\sigma}= u_\sigma^\frac{4}{n-2}g$ on $M_\sigma$. A standard computation shows that
\begin{align*}
    m_\mathrm{ADM}(\mathcal E, \tilde g_{\sigma})
    -  m_\mathrm{ADM}( \mathcal E, g)
    &=\lim_{r\to\infty}\frac{-2}{(n-2)\omega_{n-1}} \int_{|x|=r} \nu_g(u_{\sigma})\, d\mu_{S_r, g} \\
    &=\frac{-2}{(n-2)\omega_{n-1}} \int_{M_\sigma} \Delta_g u_{\sigma}\, d\mu_{g} \\
    &=\frac{-1}{2(n-1)\omega_{n-1}}\int_M V u_{\sigma}\,d\mu_g,
\end{align*}   
where the inner boundary term of the integration by parts vanishes due to the Neumann condition for $u_\sigma$.
The formula \eqref{massformula} now follows because $m_\mathrm{ADM}(\mathcal E, \tilde g_\sigma)\to m_\mathrm{ADM}(\mathcal E, \tilde g)$ by Lemma \ref{mass_conv}, and the corresponding integrals obviously as well since $u_\sigma \to u $ uniformly on compact sets.  \end{proof}

\section{Proof of the density theorem, Theorem \ref{DensityTheorem}}\label{sec:density}

Let $(M^n,g, \mathcal E)$, $(p,q)$, $\rho$, $\ve$, and $K$ be as in the statement of Theorem \ref{DensityTheorem} and Section \ref{preliminaries}.  Let $\chi_\lambda(x)=\chi(x/\lambda)$ be the family of cutoff functions defined below Lemma \ref{EllipticEstimate}. Define $g_\lambda=\chi_\lambda g+(1-\chi_\lambda)\overline g$, where $\overline g$ is the background flat metric on $\mathcal E$. We may take $\sigma_0$ to be any positive regular value of $\rho$ such that $K\subset M_{\sigma_0}$. 

 For $0<\sigma<\sigma_0$ a regular value of $\rho$ we consider the conformal Laplace-type equation 
\begin{align}
    -a\Delta_\lambda u_{\lambda,\sigma}  + (R_\lambda -\chi_\lambda R_g)u_{\lambda,\sigma} &=0 \quad\text{in }M_\sigma,\label{mainPDE1}\\
    \nu_g (u_{\lambda,\sigma} )&=0 \quad\text{on }\partial M_\sigma,\nonumber\\
   u_{\lambda,\sigma}&\to 1\quad\text{on $\mathcal E$},\nonumber
\end{align}
where $u_{\lambda,\sigma}:M_\sigma\to\Bbb R$ and $\Delta_\lambda$ and $R_\lambda$ refer to the Laplacian and scalar curvature of the metric $g_\lambda$, respectively. 
Setting $v_{\lambda,\sigma}=u_{\lambda,\sigma}-1$, we therefore solve
\begin{align}
    -a\Delta_\lambda v_{\lambda,\sigma}  + (R_\lambda -\chi_\lambda R_g) v_{\lambda,\sigma} &=- (R_\lambda -\chi_\lambda R_g) \quad\text{in }M_\sigma,\label{mainPDE2}\\
   \nu_g (v_{\lambda,\sigma}) &=0 \quad\text{on }\partial M_\sigma,\nonumber\\
   v_{\lambda,\sigma}&\in W^{2,p}_{-q'}(M_\sigma)\nonumber
\end{align}
as in Propositions \ref{Existence} and \ref{ConfFactConst}, where $q'\in(\frac{2n}{n-2},q)$. (The reason for the change from $q$ to $q'$ will become apparent in the proof.) To verify the hypotheses of these propositions, we first require a basic integration lemma. 

\begin{lem} \label{dyadic}
Let $f\in L^p_{-q-2}$, where $p>n$ and $q>\frac{n-2}{2}$. Let $A_i$ denote the dyadic annulus $2^i\le |x|\le 2^{i+1}$. For any $s\in [\frac{2n}{n+2},p]$ and $i$ sufficiently large there exists a constant $C$ independent of $f$ and $i$ such that
\[\|f\|_{L^s(A_i)}\le C2^{-\eta i}\|f\|_{L^p_{-q-2}},\]
where $\eta = q-\frac{n-2}{2}>0$.
\end{lem}
\begin{proof}
 We first use H\"older's inequality to estimate 
 \begin{align*}
     \left(\int_{A_i}|f|^s\,dx\right)^{\frac ps} &\le \vol(A_i)^{\frac{p}{s}-1}\int_{A_i}|f|^p\,dx\\
     &\le C (2^i)^{n(\frac ps-1)} \int_{A^i}|f|^p\,dx.
 \end{align*}
 Now 
 \[(2^i)^{n(\frac ps-1)}=\left((2^{i})^{\frac ns-(q+2)}\right)^p \cdot \left((2^{i})^{q+2}\right)^p(2^i)^{-n},\]
 so that 
  \begin{align*}
     (2^i)^{n(\frac ps-1)} \int_{A^i}|f|^p\,dx &=\left((2^{i})^{\frac ns-(q+2)}\right)^p\int_{A_i}|(2^{i})^{q+2}f|^p\,\frac{dx}{(2^i)^n}\\
     &\le C \left((2^{i})^{\frac ns-(q+2)}\right)^p \int_{A_i}||x|^{q+2}f|^p\,\frac{dx}{|x|^n}\\
     &\le C (2^{-i\eta})^p \|f\|_{L^p_{-q-2}}^p.\qedhere
 \end{align*}
\end{proof}

\begin{lem}
Assume the hypotheses and notation of Theorem \ref{DensityTheorem}. There exist constants $C$ and $I$ with the following property. Let $\lambda_i=2^i$. Then for any $i\ge I$ and $\sigma\in (0,\frac{\sigma_0}{2})$, there exists a unique solution $v_{\lambda_i,\sigma}\in W^{2,p}_{-q'}(M_\sigma)$ of \eqref{mainPDE2} satisfying 
\[\lim_{i\to\infty}\left(\sup_{M_\sigma}|v_{\lambda_i,\sigma}|+\|v_{\lambda_i,\sigma}\|_{W^{2,p}_{-q'}(M_{\sigma_0})}\right)=0.\]
\end{lem}

\begin{proof}
 We apply Proposition \ref{Existence} with $g=g_{\lambda_i}$ and $V=-f=R_{\lambda_i} -\chi_{\lambda_i} R_g$. We only need to check that the relevant estimates are satisfied. To obtain the required weighted $L^p$ smallness for the potential $V$, we will be required to lower the decay rate from $q$ to $q'$.
 
 First, we claim that the constants in the pointwise decay for $g_\lambda$, \eqref{pointwise}, are uniformly bounded. Indeed, in coordinates,
 \[g_{\lambda\,ij}-\delta_{ij}=\chi_\lambda (g_{ij}-\delta_{ij})\]
 and 
 \[\partial_kg_{\lambda\,ij}=\partial_k \chi_\lambda (g_{ij}-\delta_{ij})+\chi_\lambda \partial_kg_{ij}.\]
 We clearly have 
 \begin{equation}\label{glampointwise}|g_{\lambda\,ij}-\delta_{ij}|\le C_g|x|^{-q}\end{equation}
independently of $\lambda$. For the derivative, note that $|\partial\chi_\lambda|\approx \lambda^{-1}\approx |x|^{-1}$ on $\spt(\partial \chi_\lambda)$, so 
\begin{equation}|\partial_k g_{\lambda\,ij}|\le C|x|^{-1}\cdot C_g |x|^{-q}+C_g |x|^{-q-1},\label{derivativepointwise}\end{equation}
which is also uniform in $\lambda$. 

We now claim that 
\begin{equation}\lim_{\lambda\to\infty}\|R_\lambda -\chi_\lambda R_g\|_{L^p_{-q'-2}}=0.\label{H1}\end{equation}
The unweighted estimates for $R_{\lambda_i}-\chi_{\lambda_i}R_g$ then follow from Lemma \ref{dyadic} with $s=\frac n2$ and $\frac{2n}{n+2}$:
\[\|R_{\lambda_i}-\chi_{\lambda_i}R_g\|_{L^\frac{2n}{n+2}\cap L^\frac n2}\le C2^{-\left(q'-\frac{n-2}{2}\right)i}\|R_{\lambda_i} -\chi_{\lambda_i} R_g\|_{L^p_{-q'-2}} \underset{i\to 0}{\to} 0.\]

A proof of \eqref{H1} can be found in \cite{Kuwert}, for example, but we give the argument here for completeness. It is easiest to compute $R_\lambda -\chi_\lambda R_g$ and note that the worst decaying second derivatives of $g$ cancel out. Indeed, we have 
\[R_\lambda = \partial_j(\partial_i g_{\lambda\, ij}-\partial_j g_{\lambda\,ii})+O((g_\lambda-\delta)\partial^2 g_\lambda)+O((\partial g_\lambda)^2),\]
with summation over $i$ and $j$ implied in the first term on the right. This term can be expanded as 
\begin{align*}\partial_j(\partial_i g_{\lambda\, ij}-\partial_j g_{\lambda\,ii})&=\chi_\lambda \partial_j(\partial_i g_{ij}-\partial_j g_{ii})+\partial\chi_\lambda \partial g+\partial^2\chi_\lambda (g-\delta)\\&=\chi_\lambda \partial_j(\partial_i g_{ij}-\partial_j g_{ii})+O(|x|^{-q-2}).\end{align*} For first derivatives we again have $|\partial g_\lambda|\le C|x|^{-q}$ and for second derivatives 
\begin{align*}\partial^2 g_\lambda &=\partial^2\left(\chi_\lambda (g-\delta)\right)= \partial^2\chi_\lambda (g-\delta)+\partial\chi_\lambda \partial(g-\delta)+\chi_\lambda \partial^2g\\ &=\chi_\lambda\partial^2g+O(|x|^{-q-2}).\end{align*} 
Here all instances of Landau notation occur with implied constants independent of $\lambda$. Putting everything together, we find that 
\[|R_\lambda  - \chi_\lambda R_g|\le C|x|^{-q-2}+C|x|^{-q}|\partial^2 g|\]
when $\lambda\le |x|\le 2\lambda$ and this difference vanishes everywhere else. For the second term, we have 
\[\||x|^{-q}\partial^2 g\|_{L^p_{-q-2}(\{\lambda\le |x|\le 2\lambda\})}\le C\lambda^{-q}\to 0,\]
so the same is true with decay $q'<q$. Now $q'$ becomes crucial for the first term, since we have 
\[\||x|^{-q-2}\|_{L^p_{-q'-2}(\{\lambda\le |x|\le 2\lambda\})}^p\le C\int_\lambda^{2\lambda}r^{-p(q-q')-1}\,dr\le C\lambda^{-p(q-q')}\to 0.\]
This completes the proof of the claim and hence the lemma follows.
\end{proof}

\begin{proof}[Proof of Theorem \ref{DensityTheorem}]
For $\lambda_i =2^i$, let $u_{\lambda_i}$ be the functions whose existence is guaranteed by  applying Proposition \ref{ConfFactConst} to the metric $g_{\lambda_i}$ with $V=R_\lambda -\chi_\lambda R_g$, and let $\tilde g_i=u_{\lambda_i}^\frac{4}{n-2}g_{\lambda_i}$. By construction, these are harmonically flat and
$\ve$-close to $g$ in the asymptotic topology and on $K$. We first check part (i) of Definition \ref{EpsClose}. This follows from smallness of $g_{\lambda_i}-g$ and $u_{\lambda_i}-1$ in $W^{2,p}_{-q}$. As the second claim is a part of the package in Proposition \ref{ConfFactConst}, we only need to prove the first. First, we observe that 
\[\|g_\lambda-g\|_{W^{2,p}_{-q}}\le C\|\chi_\lambda-1\|_{C^2}\|g-\delta\|_{W^{2,p}_{-q}}\le C.\]
Then we note that $\spt(g_\lambda-g)\subset\{|x|\ge \lambda\}$, which implies
\[\|g_\lambda-g\|^p_{L^p_{-q'-2}}=\int_{|x|\ge \lambda}|x|^{-(q-q')}||x|^q|g_\lambda-g|^p\,\frac{dx}{|x|^n}\le C\lambda^{-(q-q')}\to 0.\]
Similar considerations apply to the derivatives and hence part (i) is proved. 
 By~\eqref{conformalformula}, the scalar curvature is given by
\[R_{\tilde g_i}=\chi_{\lambda_i}R_g u_{\lambda_i}^\frac{4}{n-2}.\]
From this we see that $R_{\tilde g_i}(x)\ge 0$ whenever $R_g(x)\ge 0$. We now prove part (ii). We have
\begin{align*}
    R_{\tilde g_i}-R_g&=\chi_{\lambda_i}R_g u^\frac{4}{n-2}_{\lambda_i}-R_g\\
    &=\chi_{\lambda_i}(u^\frac{4}{n-2}_{\lambda_i}-1)R_g + (\chi_\lambda-1)R_g.
\end{align*}
For the first term, we estimate
\[\int_\mathcal E \chi_{\lambda_i}|u^\frac{4}{n-2}_{\lambda_i}-1||R_g|\le  \sup|u^\frac{4}{n-2}_{\lambda_i}-1|\int_\mathcal E |R_g|\to 0. \]
For the second term, we have 
\[\int_\mathcal E |\chi_\lambda-1||R_g|\le \int_{|x|\ge \lambda}|R_g|\to 0\]
by elementary measure theory. Part (iii) now follows from parts (i) and (ii) together with Lemma \ref{mass_conv}.

 Since $K$ is compact, the cutoff region $|x|\ge 2^i$ misses $K$ for $i$ sufficiently large and there exists constant $C$ such that $\sup_K|R_g|\le C$. It follows that 
 \[\sup_K|R_{\tilde g_i}-R_g|\le C\sup_K|1-u_{\lambda_i}^\frac{4}{n-2}|=o(1)\]
 as $i\to\infty$, where $o(1)$ follows from the estimate~\eqref{2d} for $v_{\lambda, \sigma}$. 
\end{proof}

\section{Pushing the scalar curvature up and down}\label{UpDown}
  
  In this section we explicitly describe a well-known mechanism for increasing or decreasing mass by making appropriate conformal changes. The precise statements, which we prove are valid in the context of incomplete manifolds, will be required in our proof of rigidity in the positive mass theorem, and in Corollary \ref{CZ1}. 
  
\begin{prop}[Pushing down]\label{pushingdown}
 Let $(M^n,g)$ be an asymptotically flat manifold of Sobolev type $(p,q)$. Suppose $R_g>0$ somewhere on $M$. For any $\ve>0$, there exists a $(p,q)$ Sobolev asymptotically flat metric $\tilde g$ on $M$ which is $\ve$-close to $g$ in the asymptotic topology, with $\spt(R_{\tilde g}^-)=\spt(R_g^-)$ and 
 \[m_\mathrm{ADM}(\mathcal E, \tilde g)<m_\mathrm{ADM}(\mathcal E, g).\]
\end{prop} 
\begin{proof}
Let $B$ be a ball on which $R_g>0$. Let $\eta$ be a smooth cutoff function for $B$ such that $0<\eta < 1$ on $B$ and $\eta=0$ on $M\setminus B$. For $\delta>0$ we now solve the equation 
\[-a\Delta_g u_\delta +\delta \eta R_g u_\delta=0\]
for $u_\delta-1\in W^{2,p}_{-q}(M)$ in the sense of Propositions \ref{Existence} and \ref{ConfFactConst}. The relevant norms are $O(\delta)$, so for $\delta$ sufficiently small we obtain a unique solution of this equation with the desired asymptotic behavior. Our previous computation shows that
\[R(u_\delta^\frac{4}{n-2}g)= (1-\delta \eta)R_g u_\delta^{-\frac{4}{n-2}}.\]
Since $1-\delta \eta>0$, the sign of the scalar curvature remains pointwise unchanged. Finally, by~\eqref{massformula}, we have
\[m_\mathrm{ADM}(\mathcal E, \tilde g_\delta)-m_\mathrm{ADM}(\mathcal E, g)= -\frac{1}{2(n-1)\omega_{n-1}}\int_M \delta \eta R_gu_\delta\,d\mu_g <0,\]
so the mass strictly decreases.
\end{proof}

\begin{prop}[Bumping up]\label{bumpingup}
 Let $(M^n,g)$ be an asymptotically flat manifold of Sobolev type $(p,q)$ with nonnegative scalar curvature on $\mathcal E$. Let $f:\Bbb R\to [0,1]$ be an exponentially decreasing smooth function with $f(x)>0$ for $x>2r_0$ and $f$ vanishing on $M\setminus\mathcal E$. For sufficiently small $\ve>0$, depending only on $f$ and $\|g-\delta\|_{W^{2,p}_{-q}(\mathcal E)}$, there exists a $(p,q)$ Sobolev asymptotically flat metric $\tilde g$ which is $\ve$-close to $g$ in the asymptotic topology and satisfies $\spt(R_{\tilde g}^-)\subset \spt(R_g^-)$ and $R_{\tilde g}(x)\ge c f(|x|)$ for $|x|\ge 2r_0$ and a constant $c>0$ depending only on $\ve$ and the other stated parameters.
\end{prop}
\begin{proof}
For $\delta>0$ we solve the equation 
\[-a\Delta_g u_\delta -\delta f u_g=0\]
in the sense of Propositions \ref{Existence} and \ref{ConfFactConst}. 
The norms are again $O(\delta)$, so we can solve the equation with the desired asymptotics, for $\delta$ small depending on $f$, the geometry, and $\ve$. The scalar curvature of the conformal metric is 
\[R(\tilde g_\delta)=(R_g+\delta f)u^{-\frac{4}{n-2}}_\delta.\] For $\delta$ small depending only on the allowable parameters, $u_\delta \le 2$. It follows that $R(\tilde g_\delta)\ge \delta  2^{-\frac{4}{n-2}}f$. Finally, we note that the mass strictly increases in this process. 
\end{proof}

\section{Proofs of the positive mass theorems}\label{sec:PMT}

\begin{proof}[Proof of Theorem \ref{thm1}]
For $(M,g)$ as in the statement of the theorem, suppose \hbox{$m_\mathrm{ADM}(\mathcal E, g)<0$}. We approximate $g$ with metrics $\tilde g$ as in the density theorem. For $\ve$ sufficiently small, $(M,\tilde g)$ satisfies the hypotheses of the corresponding theorem for asymptotically Schwarzscild manifolds, \cite[Theorem 1.7]{LUY21}.\footnote{The constant in \cite{LUY21} is worse than the one cited here. The better constant follows from a more carefully constructed $\mu$-bubble function $h$, as presented in in Section \ref{MOTS} below.} But for $\ve$ small we still have $m_\mathrm{ADM}(\mathcal E, \tilde g)<0$, which gives a contradiction. 

Finally, we must also rule out $m_\mathrm{ADM}(\mathcal E, g)=0$. By an immediate application of Proposition \ref{pushingdown}, we can find approximating metrics $\tilde g$ with negative mass and satisfying the hypotheses of the theorem, thereby contradicting the previous paragraph.
\end{proof}

We next prove the positive mass theorem for complete manifolds, with rigidity. 

\begin{proof}[Proof of Theorem \ref{pmt}]
Let $(M,g)$ satisfy the hypotheses of the theorem.
We first prove the inequality $m_\mathrm{ADM}(\mathcal E, g)\ge 0$. Suppose otherwise. Then by our density theorem (Theorem~\ref{DensityTheorem}), we can find $\tilde{g}$ that is harmonically flat outside a compact set and has negative mass. In particular, $\tilde g$ is asymptotically Schwarzschild, so this 
contradicts the known positive mass inequality for complete manifolds with an asymptotially Schwarzschild end \cite[Theorem 1.2]{LUY21}.

We now prove rigidity, which roughly follows the standard conformal approach, except that we use our new results from Sections~\ref{preliminaries} and~\ref{UpDown}. Assume $m_\mathrm{ADM}(\mathcal E, g)= 0$.
First we claim that $g$ is scalar-flat. Otherwise we can use Proposition \ref{pushingdown} to obtain a new metric which still has nonnegative scalar curvature but has negative mass, contradicting the positive mass inequality that we already proved.
Next, we show that $g$ is Ricci-flat as well. 

Let $\eta$ be a compactly supported cutoff function and $t\in \Bbb R$. We consider the deformed metrics $g_t=g+ t\eta \Ric_g$. For $t$ sufficiently small, these will indeed be Riemannian metrics and will satsify the analytic hypotheses of Section \ref{preliminaries} uniformly. We solve the equations 
\[-a\Delta_{g_t}u_t+R_{g_t}u_t=0\]
to obtain a scalar-flat metric  $\tilde g_t:=u_t^\frac{4}{n-2}g_t$. One can see that this is possible by Propositions \ref{Existence} and \ref{ConfFactConst}, for $|t|$ sufficiently small. 
By the positive mass inequality for $\tilde{g}$, we know that    $\frac{m_\mathrm{ADM}(\mathcal E, \tilde g_t)}{t} \ge 0
  $ for $t>0$ and 
   $\frac{m_\mathrm{ADM}(\mathcal E, \tilde g_t)}{t} \le  0$ for $t<0$.
Then by~\eqref{massformula}, we can see that
\[-\frac{1}{2(n-1)\omega_{n-1}}\lim_{t\to 0}\frac{1}{t}\int_M R_{g_t}u_t\,d\mu_{g_t}=   \lim_{t\to 0} \frac{m_\mathrm{ADM}(\mathcal E, \tilde g_t)}{t} =0.
   \]
Using the dominated convergence theorem and the calculation as in \cite[page 96]{Lee}, we can also see that
\[\lim_{t\to 0}\int_M \frac{R_{g_t}}{t}u_t\,d\mu_{g_t}=\int_M\eta|\Ric_g|^2,\]
Combining the two equalities above, we see that $\Ric_g=0$ on the support of $\eta$. Since $\eta$ was arbitrary, this implies $g$ is Ricci-flat. 

Now we show that $(M,g)$ has only one end. If $M$ has a second end, then it contains a geodesic line $\gamma$ which goes out to infinity along the asymptotically flat end \cite{Petersen}. By the Cheeger--Gromoll theorem, $(M,g)$ splits isometrically along this line as $(\Bbb R\times N, dt^2 +h)$. Concretely, there is a smooth function $f:M\to\Bbb R$ whose level sets foliate $M$ and are all isometric to $(N,h)$, the isometry being the gradient flow of $f$. If $N$ is not flat, then there exists a ball $B\subset M$ such that 
\begin{equation}\label{7a}\int_{B}|\Rm_g|^p\,d\mu_g>0.\end{equation}
We flow $B$ along the gradient flow of $f$ in the direction of the asymptotically flat end. Since the gradient flow is an isometry, the distance between the ball and $\gamma$ is unchanged, as is the integral \eqref{7a}. But $\Rm_g\in L^p(\mathcal E)$, so the integral \eqref{7a} must $\to 0$ as the ball goes further and further out, which is a contradiction. 

Since $N$ is flat, it is either isometric to $\Bbb R^{n-1}$ or a nontrivial quotient, in which case $\pi_1(N)\ne 0$. We must rule out the latter case. Let $\ell$ be a homotopically nontrivial loop in $N$. Then $\gamma$ is contained in some compact ball and we can push it along the line $\gamma$. It enters $\mathcal E\approx \Bbb R\times S^{n-1}$ and can be contracted, a contradiction. 

We conclude that $(M,g)$ cannot have two or more ends. We now show that $(M^n,g)$ is flat. Take a point $p\in M$ and consider the volume ratio 
\[V(r)=\frac{\vol_g(B_r(p))}{\alpha_n r^n}.\]
We have 
\[\lim_{r\to 0}V(r)=\lim_{r\to\infty}V(r)=1.\]
The limit as $r\to 0$ is true for any manifold, the limit as $r\to\infty$ is a consequence of one-endedness and asymptotic flatness. A careful proof can be found in \cite[Lemma 2.6]{LiYu}. Now the Bishop--Gromov volume comparison theorem implies $V(r)=1$ for any $r$ and hence $(M,g)$ is flat. 
\end{proof}

\begin{proof}[Proof of Corollary \ref{CZ1}]
Apply Proposition \ref{bumpingup} with an arbitrarily chosen $f$ and $\ve$ sufficiently small that $m_\mathrm{ADM}(\mathcal E, \tilde g)<0$. Then we know that the hypotheses of Theorem \ref{thm1} are violated in $N_D^{\tilde g}(\mathcal E)$ for $\tilde g$, where $D$ is such that \eqref{largeness} is satisfied on some annular region where $f>0$. However, $\tilde g$ does not have any new points of negative scalar curvature and no new incomplete points, so one of the hypotheses must be violated for $g$ in $N^{\tilde g}_D(\mathcal E)$ as well.  Since $g$ and $\tilde g$ are uniformly equivalent, $N_D^{\tilde g}(\mathcal E)\subset N^g_{D'}(\mathcal E)$ for some $D'$ close to $D$.
 \end{proof}

\section{Analogy between $\mu$-bubbles and MOTS} \label{MOTS} 

We begin by recalling Gromov's $\mu$-bubbles \cite{G96,G18}: 

\begin{defn}\label{mububble}
Let $(M,g)$ be a Riemannian manifold and $h$ (called the \emph{potential}) a continuous extended real-valued function on $M$ with the following properties: 

\begin{enumerate}
    \item Let $M_0$ be the closure of $\{|h|<\infty\}$ in $M$. Then $h$ is smooth on $M_0$.
    
    \item $\partial M_0=\partial M_+\cup \partial M_-$, where $\partial M_\pm$ are \emph{nonempty} smooth closed embedded hypersurfaces and $h|_{\partial M_\pm}=\pm\infty$. 
\end{enumerate}
Let $\Omega_0$ be a Caccioppoli set in $M_0$ which contains $\partial M_+$.
For $\Omega$ another such Caccioppoli set, we define 
\[\mathcal F(\Omega)= \mathcal H^{n-1}_g(\partial^*\Omega)-\int_{M_0}(\chi_\Omega-\chi_{\Omega_0})\,h\,d\mu_g.\]
We say that $\Omega$ is a \emph{$\mu$-bubble} if it is a critical point of this functional under variations satisfying $\Omega\symD\Omega_0\subset\subset M_0$, or equivalently, if $\partial\Omega$ has prescribed mean curvature $h$ (with the normal oriented pointing towards $\partial M_-$). \emph{Stable} and \emph{minimizing} $\mu$-bubbles are defined in the obvious way. 
\end{defn}

We also recall the following existence and regularity result \cite[Proposition 2.1]{Z21} (see also \cite[Proposition 12]{CL20}). 

\begin{lem}\label{Zhu}
If $2\le n\le 7$, nonempty stable $\mu$-bubbles exist, and are smooth.  
\end{lem}
\begin{proof}[Sketch of Proof]
Since $\partial M_\pm$ are closed, smooth hypersurfaces, any foliation of a tubular neighborhood has uniformly bounded mean curvature. Therefore, $H<h$ near $\partial M_+$ (with the normal pointing away from $\partial M_+$) and $H>h$ near $\partial M_-$ (with the normal pointing away from $\partial M_-$). These surfaces act as barriers, and the BV compactness and regularity theory guarantees the existence and regularity of a minimizing, and hence stable, $\mu$-bubble.
\end{proof}

The utility of stable $\mu$-bubbles is explained by the following computation:

\begin{lem}
If $\Omega$ is a smooth stable $\mu$-bubble with boundary $\Sigma$, then $\Sigma=\partial\Omega$ satisfies the stability inequality 
\begin{equation}\label{mustab}
    \int_{\Sigma}|\nabla\varphi|^2 +\tfrac 12 R_\Sigma\varphi^2-\tfrac 12\left(R_g+\tfrac{n}{n-1}h^2+2\nu(h)\right)\varphi^2\,d\mu_{\Sigma} \ge 0,
\end{equation}
for every $\varphi\in C^1(\Sigma)$. 
\end{lem}
\begin{proof}
The second variation of the $\mu$-bubble functional is given by (see \cite[Proposition 2.3]{LUY21})
\begin{equation}\label{mustability}
\int_\Sigma |\nabla \varphi|^2-\tfrac 12\left(R_g-R_\Sigma+|A|^2\right)\varphi^2 -\tfrac 12\left(h^2+2\nu(h)\right)\varphi^2 \,d\mu_{\Sigma}\ge 0.
\end{equation}
We then write $|A|^2= \tfrac 1n H^2+|\mathring A|^2 = \tfrac 1n h^2+|\mathring A|^2$, insert this into the second variation, and rearrange to obtain \eqref{mustab}.
\end{proof}

From this stability inequality, we see that if 
\begin{equation}\tag{$\star$}\label{mububblecondition}
R_g+\tfrac{n}{n-1}h^2-2|\nabla h|\ge 0,
\end{equation}
then $\Sigma$ admits a metric of nonnegative scalar curvature. (And a strict inequality implies $\Sigma$ admits positive scalar curvautre.)  
We will refer to this inequality as condition \eqref{mububblecondition}.\footnote{In \cite{LUY21}, the condition is written as $R_M+h^2-2|\nabla h|>0$. This implies the current condition since $\frac{n}{n-1}>1$. The difference comes from keeping the trace part of $|A|^2$ versus just throwing it away.} 

To explain the connection with marginally outer trapped surfaces, we first recall their definition. Recall that $(M, g, k)$ is an \emph{initial data set} if $(M, g)$ is a Riemannian manifold equipped with a symmetric $(0,2)$-tensor field $k$, and that the \emph{mass density} $\mu$ (a scalar) and the \emph{momentum density} $J$ (a vector field) are defined by
\begin{align*}
    \mu &=\tfrac 12\left(R_g-|k|_g^2+(\tr_gk)^2\right)\\
   J^i &= (\Div_g k)^i -\nabla^i (\tr_g k).
\end{align*}
We say that $(M,g,k)$ satisfies the \emph{dominant energy condition} (\emph{DEC}) if $\mu \ge |J|_g$. 

 A hypersurface $\Sigma$ in an initial data set $(M^n,g,k)$ with distinguished choice of normal $\nu$ is a \emph{marginally outer trapped surface (MOTS)} if 
\[\theta^+=H+P=0,\]
where $H$ is the mean curvature and $P=\tr_\Sigma k=(g^{ij}-\nu^i\nu^j)k_{ij}$. MOTS do not satisfy a variational criterion, but there is a naturally associated stability operator \cite{AMS05}
\[L=-\Delta \varphi+2\langle W_\Sigma,\nabla u\rangle +(\Div_\Sigma W_\Sigma-|W_\Sigma|^2+Q_\Sigma)\varphi,\]
where $W_\Sigma=k(\nu,\cdot)$ restricted to $T\Sigma$, 
\[Q_\Sigma=\tfrac 12 R_\Sigma-\mu-\langle J,\nu\rangle -\tfrac 12 |k_\Sigma+A|^2,\]
where $k_\Sigma$ is $k$ restricted to $T\Sigma$. A MOTS is \emph{stable} if $\lambda_1(L)\ge 0$.\footnote{The operator $L$ is not self-adjoint, but it still has a real principal eigenvalue and an associated real positive eigenfunction. See \cite{AMS05}.} Closely related is the \emph{symmetrized MOTS stability operator} of Galloway--Schoen \cite{GS06}
\[L_\mathrm{sym}=-\Delta_\Sigma +Q_\Sigma.\] Notably, Galloway and Schoen showed that $\lambda_1(L)\le \lambda_1(L_\mathrm{sym})$ for any MOTS. 
We now explain a relationship between stable MOTS and stable $\mu$-bubbles.

\begin{defn}
Given $(M,g,h)$ as in Definition \ref{mububble}, we define a data set $(M_0,g,k_h)$ by taking $M_0=\overline{\{|h|<\infty\}}$ and setting 
\[k_h=-\frac{h}{n-1}g.\]
\end{defn}
The first observation is that if $\Sigma$ is a hypersurface in $M_0$, then $P=-h$, so that 
\[\theta^+=H-h,\]
and thus $\Omega$ is a $\mu$-bubble with respect to $(M, g, h)$ if and only if $\partial\Omega$ is a MOTS with with respect to $(M, g, k_h)$.
In fact, more is true.
\begin{prop} \label{stableMOTS}
 Let $(M,g,h)$ and $\Sigma=\partial \Omega$ smooth be as in Definition \ref{mububble}. Then:
 \begin{enumerate}
 \item \label{part1}$\Sigma$ is a stable $\mu$-bubble if and only if $\Sigma$ is a stable MOTS with respect to $(g,k_h)$. 
 
 \item \label{part2}$(g,h)$ satisfies condition~\eqref{mububblecondition} if and only if $(g,k_h)$ satisfies the dominant energy condition. 
 \end{enumerate}
\end{prop}
\begin{proof}
The first thing to note is that $W_\Sigma=0$ for our choice of $k$. Therefore, $L=L_\mathrm{sym}$ in this setting. To compute $Q_\Sigma$, we note the following, which the reader may easily verify:
\begin{align*}
    \tr k&=-\tfrac{n}{n-1}h,\\
    |k|^2&=\tfrac{n}{(n-1)^2}h^2,\\
    \nabla^j k_{ij}&=-\tfrac{1}{n-1}\nabla_i h,\\
    \nabla_i \tr k&=-\tfrac{n}{n-1}\nabla_i h.
\end{align*}
It follows that 
\begin{align*}
    2\mu&=R_g +\tfrac{n}{n-1}h^2\\
    J^i&= \nabla^ih.
\end{align*}
Putting these together yields 
\begin{align*}
2(\mu+\langle J,\nu\rangle)&=R_g+\tfrac{n}{n-1}h^2+2\langle \nabla h,\nu\rangle,\\
2(\mu -|J|)&=R_g+\tfrac{n}{n-1}h^2-2|\nabla h|,\end{align*}
The second equation verifies part (\ref{part2}). To complete the proof of part (\ref{part1}), if $\Omega$ is a $\mu$-bubble, then $H=h$, and we can see that 
\begin{align*}
|k_\Sigma+A|^2&= \tfrac{1}{n-1}h^2 -\tfrac{2}{n-1}hH + |A|^2\\
&= -\tfrac{1}{n-1}h^2  + |A|^2.
\end{align*}
Thus
\[Q_\Sigma=\tfrac 12 R_\Sigma - \tfrac 12 R_g - \tfrac 12 |A|^2- \tfrac 12 h^2- \langle \nabla h,\nu\rangle, \]
and we can explicitly see that the MOTS stability inequality  $\lambda(L_{\mathrm{sym}})\geq 0$ is the same as the $\mu$-bubble stability inequality~\eqref{mustability}.
\end{proof}

Using this, we observe that Lemma \ref{Zhu} can be seen as a special case of the existence theorem for stable MOTS. 

\begin{proof}[Proof of Lemma \ref{Zhu} using MOTS]
Since $h\to\pm\infty$ as we approach $\partial_\pm M$, hypersurfaces foliating a small neighborhood of $\partial M_+$ and $\partial M_-$ will be strictly trapped and untrapped, respectively. Hence we may apply the MOTS existence theory of L.~Andersson, M.~Eichmair, and J.~Metzger \cite{E09, Andersson-Metzger, AEM} to find a nontrivial MOTS in $M_0$. By part (\ref{part1}) of Proposition \ref{stableMOTS}, this gives us the desired nontrivial stable $\mu$-bubble.  
\end{proof}

This is a more complicated proof of the lemma, but it illustrates the general principle. We now use this viewpoint to give a surprising new proof of the quantitative shielding theorem, Theorem \ref{thm1}. We will also use it to prove Theorem \ref{CZ2}.

\begin{proof}[Proof of Theorem \ref{thm1}]
Assume the hypotheses of Theorem \ref{thm1}. In particular, we have $U_2\subset U_1\subset U_0$ on which $R_g\ge0$, and furthermore, on $\overline{U}_1\setminus U_2$, we have $R_g > \frac{4}{D_0 D_1}$ where $D_0=\dist_g(\partial U_0, U_1)$ and $D_1=\dist_g(\partial U_1, U_2)$. 

We first revisit the construction of the potential function $h$ from \cite{LUY21}.\footnote{The construction presented here is cleaner than in \cite{LUY21}, which leads to an improved constant, but the basic idea is the same.}  Let $\ve>0$ be small enough so that $\ve<\tfrac{1}{n-1}$ and 
\begin{equation}\label{epsrequirement}
    R_g>\frac{4(1+\ve)}{(D_0-2\ve)D_1}
\end{equation}
on the compact set $\overline{U}_1\setminus U_2$. Let $\rho=\dist_g(x,U_1)$ and let $\tilde\rho$ be a smoothing of $\rho$ on $U_0$ such  that $\rho$ also vanishes on $U_1$ and the following inequalities hold:
\begin{itemize}
    \item $\sup_{U_0}|\tilde\rho-\rho|<\ve$,
    \item $|\nabla\tilde\rho|<1+\ve$.
\end{itemize}
One can construct $\tilde\rho$ by mollifying the Lipschitz distance function from the $\ve/2$-neighborhood of $U_1$.  
Also, let $\vartheta$ be a smooth cutoff function such that $\vartheta=1$ on $M\setminus U_1$, $\vartheta=0$ on $U_2$, and $|\nabla \vartheta|\le (1+\ve)D_1^{-1}$ on $\overline U_1\setminus U_2$. To see that such a $\vartheta$ exists, first construct a Lipschitz function that is $1$ and $0$ on \emph{neighborhoods} of  $M\setminus U_1$ and $\overline{U}_2$, respectively, and then mollify.

Now select any $\alpha$ between $D_0-2\ve$ and $D_0-\ve$ such that the level set $\tilde\rho^{-1}(\alpha)$ is smooth. Then we define
\[h_1(x)=\begin{cases}
\frac{2}{\alpha-\tilde\rho}& \text{if }\tilde\rho(x)<\alpha\\
+\infty &\text{if }\tilde\rho(x)\ge \alpha
\end{cases}\]
and 
\[h=\vartheta h_1.\]

We claim that this choice of $h$ satisfies condition~\eqref{mububblecondition}. Note that $M_0=\overline{\{h<\infty\}}$ is contained in $U_0$, so there are three regions to analyze: $U_2$, $\overline U_1\setminus U_2$, and $U_0\setminus U_1$. The potential $h$ vanishes identically on $U_2$, so~\eqref{mububblecondition}
is trivially satisfied since $R_g\ge 0$ on $U_2$. On the region $\overline U_1\setminus U_2$, $h_1$ is a constant equal to $\frac{2}{\alpha}$, so we have
\begin{equation}\label{star-alpha}
\tfrac{n}{n-1}h^2-2|\nabla h|\ge -\frac{4(1+\ve)}{\alpha D_1}.
\end{equation}
Combining this with~\eqref{epsrequirement} and the definition of $\alpha$, we see that condition~\eqref{mububblecondition} holds in this region as well.
Finally, on $U_0\setminus U_1$, we have
\[
    2|\nabla h|=\frac{4}{(\alpha-\tilde\rho)^2}|\nabla\tilde\rho|\le \frac{4}{(\alpha-\tilde\rho)^2}(1+\ve)=(1+\ve) h_1^2\le \tfrac{n}{n-1} h_1^2
\]
Since $R_g\ge 0$ on $U_0$, we see that~\eqref{mububblecondition} holds on $U_0\setminus U_1$, and hence we have shown that it holds everywhere.

By Proposition~\ref{stableMOTS}, we observe that $(g, k_h)$ defines asymptotically flat initial data on $M_0$ that satisfies the dominant energy condition, and moreover, since $\partial M_0$ is smooth, the level sets of $h$ near $\partial M_0$, where $h$ is large, must be strictly outer trapped surfaces. The positivity of  $m_\mathrm{ADM}(\mathcal E, g)$ now follows from the \emph{spacetime positive mass theorem with boundary} \cite[Theorem 1.3]{LLU} applied to $(M_0, g, k_h)$. (Note that the nonnegativity follows from~\cite{GallowayLee} as well.) 

Finally, we note that \cite[Theorem 1.3]{LLU} requires $C^2_{-q}$ decay of the metric rather than the Sobolev decay in Definition~\ref{defAF}, 
but Theorem \ref{DensityTheorem} can be used to assume we have this decay without loss of generality when proving the \emph{nonnegativity} of mass. Then as explained in Section~\ref{sec:PMT}, positivity follows from the nonnegativity using Proposition~\ref{pushingdown}. 
\end{proof}

\begin{proof}[Proof of Theorem \ref{CZ2}] The goal of this proof is to choose $\alpha$ so that $\partial M$ is outer trapped with respect to the data set $(g,k_h)$, while maintaining the DEC. The result then follows from the spacetime positive mass theorem with boundary, as above. 

According to \eqref{star-alpha}, we can maintain DEC, which is equivalent to \eqref{mububblecondition}, by setting
\[\alpha = \frac{4(1+\ve)}{\kappa D_1}.\]
Meanwhile, $\partial M$ being outer trapped means that $H<\min_{\partial M}h$, so we compute
\[\min_{\partial M}h\ge \frac{2}{\alpha-D_0+\ve}=\frac{2}{\frac{4(1+\ve)}{\kappa D_1}- D_0+\ve}.\]
Taking $\ve\to0$ yields the desired bound appearing in Theorem \ref{CZ2}.

From here we can now invoke the spacetime positive mass theorem with boundary~\cite{GallowayLee, LLU}. Although these results assume $C^2_{-q}$ decay, we can reduce to this case using our density theorem (Theorem~\ref{DensityTheorem}).
\end{proof}

\bibliographystyle{alpha}
\bibliography{bib}
\end{document}